\documentclass{amsart}
\usepackage{amssymb}
\usepackage{amsfonts}
\usepackage{amscd}
\usepackage[all]{xypic}

\swapnumbers

\def\ordp{{\rm ord}_p}

\def\cD{{\mathcal D}}
\def\cP{{\mathcal P}}
\def\cU{{\mathcal U}}
\def\cL{{\mathcal L}}

\newtheorem{theorem}[subsection]{Theorem}

\newtheorem{cor}[subsection]{Corollary}
\newtheorem{prop}[subsection]{Proposition}

\theoremstyle{definition}

\theoremstyle{plain}

\numberwithin{equation}{subsection}

\def\boxit#1#2{\setbox1=\hbox{\kern#1{#2}\kern#1}%
\dimen1=\ht1 \advance\dimen1 by #1 \dimen2=\dp1 \advance\dimen2 by
#1
\setbox1=\hbox{\vrule height\dimen1 depth\dimen2\box1\vrule}%
\setbox1=\vbox{\hrule\box1\hrule}%
\advance\dimen1 by .4pt \ht1=\dimen1 \advance\dimen2 by .4pt
\dp1=\dimen2 \box1\relax}

\newcommand{\abs}[1]{\lvert#1\rvert}

\begin{document}
 \title{A Note on complex $p$-adic exponential fields}

\author{Ali Bleybel}

\address{Lebanese University,
 Faculty of Sciences,
 Beirut, Lebanon}

\email{bleybel@ul.edu.lb}

\begin{abstract}
In this paper we apply Ax-Schanuel's Theorem to the ultraproduct of $p$-adic fields in order to get some results towards algebraic independence of $p$-adic exponentials for almost all primes $p$. 
\end{abstract}

\maketitle

\section{Introduction}
 Let $\mathbb{Q}_p$ be the field of $p$-adic numbers, for $p$ a prime number.  
   Given an algebraic closure $\mathbb{Q}_p^{\rm alg}$ of $\mathbb{Q}_p$, it comes naturally equipped with a norm $\abs{\cdot}_p$, uniquely extending the usual norm on $\mathbb{Q}_p$. Recall that the standard normalization for $\abs{\cdot}_p$ is $\abs{p}_p=p^{-1}$.  \\
   Denote by $\mathbb{C}_p$ the completion of $\mathbb{Q}_p^{\rm alg}$ with respect to the norm $\abs{\cdot}_p$. Then $\mathbb{C}_p$ is also algebraically closed. It is called a complex $p$-adic field.

 The $p$-adic exponential map
$$\exp_p: E_p  \to \mathbb{C}_p^{\times}, x \mapsto \sum_{n=0}^\infty \frac{x^n}{n!},$$
where $E_p$ is the set $E_p= \{ x \in \mathbb{C}_p: |x|_p < p^{-\frac{1}{p-1}} \}$ (the domain of convergence of the defining  power series of the exponential) shares several properties with the complex exponential map $\exp$ (such as $\exp_p(x+y) = \exp_p (x) \exp_p(y)$, $(\exp_p(x))'=\exp_p(x)$ where $()'$ denotes the usual derivative).

There are important open problems regarding the exponential map over a non-archimedean valued field. One of these concerns the algebraic independence of the values of the exponential map at different arguments.\\
 Such issues are encapsulated in the following well-known conjecture ($p$-adic Schanuel's conjecture) \\
{\bf ($p$-SC)} Let $\bar{x} := (x_1,\dots, x_n) \in \mathbb{C}_p^n$ be an $n$-tuple of complex $p$-adic numbers satisfying the requirement $$\abs{\bar{x}}_p  := \max_{1 \leq i \leq n}\{\abs{x_i}_p \} < p^{-1/p-1}. $$
Assume that $x_1, \dots, x_n$ are $\mathbb{Q}$-linearly independent, then 
$$ \textrm{td}_{\mathbb{Q}}(x_1, \dots, x_n, \exp_p(x_1), \dots, \exp_p(x_n)) \geq n, $$
where td$_{\mathbb{Q}}$ denotes the transcendence degree of the extension $$\mathbb{Q}(\bar{x}, \exp_p(\bar{x}))/\mathbb{Q}. $$  

In the following we will denote by $\mathbb{G}$ the algebraic group $\mathbb{G}_a \times \mathbb{G}_m$, with $\mathbb{G}_a$ denoting the additive group of a field (say $\mathbb{C}_p$) and $\mathbb{G}_m$ its multiplicative group. \\
In the above statement we used the abbreviation $f(\bar{x}):= (f(x_1), \dots, f(x_n))$ for any $n$-tuple $\bar{x}$.  \\
An equivalent statement to ($p$-SC) is the following: \\
{\bf ($p$-SC)'}  Let $\bar{x} := (x_1,\dots, x_n) \in \mathbb{C}_p^n$ be an $n$-tuple of complex $p$-adic numbers satisfying $\abs{\bar{x}}_p < p^{-1/p-1}$. 
Assume that
$$ (\bar{x}, \exp_p(\bar{x})) \in V(\mathbb{C}_p), $$
 for some subvariety $V$ of $\mathbb{G}^n$ defined over $\mathbb{Q}$ (i.e. a $\mathbb{Q}$-variety), which is furthermore of dimension $<n$. Then, $x_1, \dots, x_n$ are $\mathbb{Q}$-linearly dependent, i.e. 
 $$ m_1 x_1 + \dots + m_n x_n =0, $$
 for some $m_1, \dots, m_n \in \mathbb{Q}$, not all zero. 
\par
In this paper we apply the ultraproduct construction and basic model theory in order to obtain some results in the above direction.   \\
The main Theorem can be obtained by applying Ax-Schanuel's Theorem ~\cite{A} to a non-principal ultraproduct of $\mathbb{C}_p$, and it reads as:
\begin{theorem}  \label{main}
 Let $V$ be a $\mathbb{Q}$-variety of dimension $n$ in an affine $2n$ space. Assume that for infinitely many primes $p$, $V$ has a $\mathbb{C}_p$-point of the form $(\bar{a}_p, \exp_p(\bar{a}_p))$, then 
 there exist a finite set $S(V) \subset \mathbb{P}$, and a finite set of rational tuples $\bar{\alpha}_i, i \in I$, (where $I$ is a finite set) such that for all $p \in \mathbb{P}\setminus \!S(V)$,  and for all $n$-tuples $\bar{x}_p \in E_p^n$ satisfying $(\bar{x}_p, \exp_p(\bar{x}_p)) \in V(\mathbb{C}_p)$, there is a rational linear dependence that holds for the tuple $\bar{x}_p$ of the form 
 $$ \alpha_{i,1} x_{p,1} + \dots + \alpha_{i,n} x_{p,n} =0, $$
 for some $i \in I$.  
\end{theorem}
An equivalent statement (with a geometrical flavor) is the following: \\
\textit{ Let $V$ be a $\mathbb{Q}$-variety of dimension $n$ in a $2n$-space. If, for infinitely many primes $p$, $V$ has a $\mathbb{C}_p$-point of the form $(\bar{a}_p, \exp_p(\bar{a}_p))$, then there exist a finite set $S \subset \mathbb{P}$ and a finite set of hyperplanes $H_i \subset \mathbb{A}^{n}_\mathbb{Q}, i \in I$ such that for all $p \in \mathbb{P} \setminus \! S$, we have }
$$ \forall \bar{x}_p \in E_p^n, (\bar{x}_p, \exp_p(\bar{x}_p)) \in V(\mathbb{C}_p)  \longrightarrow (\exists i \in I) (\bar{x}_p \in H_i(\mathbb{C}_p)). $$  
In the above Theorem, the order of quantifiers is essential: for each variety $V$ as above, there is a set $S(V)$ of exceptional primes (i.e. primes $p$ for which the stated implication might not hold), and such set is only dependent on the variety $V$.  By \textit{almost all primes} we mean all except a finite set of primes. 
This is to distinguish from the notion of $\cU$-\textit{almost all} (for a given ultrafilter $\cU$) which will be encountered later. A \textit{uniform} rational linear dependence is a linear dependence of the form 
  $$ m_1 x_{1,p} + \dots + m_n x_{n,p} =0, $$
  for some \textit{fixed} rationals $m_1, \dots, m_n$, not all zero. The Theorem implies in particular that, for each family of $n$-tuples $(\bar{x})_p$ as above there exists a partition of $\mathbb{P} \setminus S(V)$ into finitely many sets, on each of which the obtained linear dependence is uniform. 
  
The method of proof uses Ax's result ~\cite{A} on Schanuel's property for differential exponential fields. \\
For each variety $V \subset \mathbb{G}^n$ of dimension $n$ as above, the conclusions of Theorem \ref{main} hold for all but possibly finitely many primes belonging to some exceptional set $S(V)$. For a particular variety $V \subset \mathbb{G}^n$ having dimension $ \leq n$, and a given prime $p \notin S(V)$, the conclusion of Theorem ~\ref{main} is strictly stronger than what is given by conjecture ($p$-SC) (or, more precisely, its equivalent ($p$-SC)'). That is, according to ($p$-SC), there might exist $\mathbb{Q}$-linearly independent tuples $\bar{x}_p$ (in the domain of $\exp_p$) for which $(\bar{x}_p, \exp_p(\bar{x}_p)) \in V(\mathbb{C}_p)$ if $V$ is of dimension $n$, while this is not the case for Theorem ~\ref{main} whenever $p \notin S(V)$. This is due to the statement of Ax's Theorem, in which the weak inequality in ($p$-SC) is replaced by a strict one.

Let $\cU$ be a non-principal ultrafilter over the set $\mathbb{P}$ of prime numbers. Consider the ultraproduct
$$ \mathbb{K}_{\cU} := \prod_{p \in \mathbb{P}} \mathbb{C}_p / \cU.$$
Then $\mathbb{K}_{\cU}$ is an algebraically closed valued field (whose valuation is induced by $p$-adic valuations on each $\mathbb{C}_p$). We may define a partial exponential map on $\mathbb{K}_{\cU}$, induced by the maps $\exp_p$. As explained in ~\cite{KMS}, $\mathbb{K}_{\cU}$ can be embedded in a differential exponential valued field, to which it is possible to apply Ax-Schanuel's Theorem.  \\
 Then by an application of \L \'os' Theorem on ultraproducts, we obtain the required result.

We will consider stronger versions of these results in a forthcoming paper. Theorem \ref{main} will be proved in section \ref{Proofs}, after several preliminary sections, which contain reminders of known results concerning valued fields, ultraproducts of valued fields and other related concepts.   
\subsection*{Acknowledgments}
I am grateful to the anonymous referee for his careful reading of the manuscript, and for many valuable
comments and suggestions. 
\section{Background }
 In this section we introduce background results that will be needed in the rest of the paper.  \\
Recall that the field $\mathbb{C}_p$ is the completion (with respect to the norm $\abs{\cdot}_p$) of an algebraic closure of $\mathbb{Q}_p$, the field of $p$-adic numbers.
One may consider instead the additive valuation $\ordp$ defined on $\mathbb{C}_p$. This valuation is defined through the relation:
$$ \abs{z}_p = p^{-\ordp(z)}. $$

In ~\cite{A} J. Ax proved the following result, already conjectured by S. Schanuel:
\begin{theorem} \label{Ax} (Ax ~\cite{A})
 Let $K$ be a differential field equipped with a derivation $D$, and let $C$ be its field of constants. \\
 Let $y_1,\dots, y_n, z_1, \dots, z_n \in K^{\times}$ be such that $Dy_i = Dz_i/z_i$. \\
 Assume that the $y_i,\;  i=1, \dots, n$ are $\mathbb{Q}$-linearly independent modulo $C$, then
$$ {\rm td}_{C} (y_1,\dots, y_n, z_1, \dots,z_n) \geq n+1. $$
\end{theorem}
Recall that a \textit{derivation} over a (commutative) field $K$ is a map $D: K \to K$ satisfying additivity ($D(x+y) = Dx + Dy$) and Leibniz rule ($D(xy)=xDy +yDx$). The {\it field of constants for $D$} is the set of $x \in K$ for which $Dx=0$. Using additivity and Leibniz rule, one can see that $C$ is indeed a subfield of $K$. \\ 
In ~\cite{KMS} this result was restated as follows:
\begin{theorem} \label{Theorem A}
  Let $y_1, \dots,y_n,z_1, \dots,z_n \in K^\times$ be such
that $Dy_k = \frac{Dz_k}{z_k} \mbox{ for } k= 1,\dots,n$. \\
 If ${\rm td}_CC(y_1,\dots,y_n, z_1, \dots,z_n) \leq n$, then $\sum_{i=1}^n
m_iy_i\in C$ for some $m_1, \dots, m_n \in \mathbb{Q}$ not all zero.
\end{theorem}
A corollary of the above is given by (this is essentially Corollary 3 in ~\cite{KMS}): 
\begin{prop} \label{ASP}
	 Let $(K,\exp)$ be a partial differential exponential field (that is, a field equipped with a partial exponential map $\exp$, satisfying $D\exp(x) = \exp(x)Dx$), with a field of constants $k$. Then, for any $n$-tuple $\bar{x} := (x_1,\dots,x_n) \in K^n$ of elements of $K$, where $x_1, \dots, x_n$ belong to the domain of the exponential map. \\
	If $(\bar{x}, \exp (\bar{x}))  \in V(K)$ for some algebraic variety $V$ of dimension $n$ with rational coefficients, 
	 $V \subset \mathbb{G}^n$, 
	then $\sum_{i=1}^n m_i x_i \in k$, for some $m_1, \dots, m_n \in \mathbb{Q}$ not all zero.
\end{prop}
\subsection{Language and Logical Setting} \label{lang}
Let $L=(+, -, \cdot, (\,)^{-1}, 1,0)$ be the language of fields, with the standard interpretation of the symbols involved.
We consider the expansion $\cL$ of $L$:
$$ \cL = L \cup \{R, \textrm{Exp}\}, $$
where $R$ is a unary predicate symbol while $\textrm{Exp}$ is a function symbol (to be interpreted as an exponential map $K \to K^\times$, with $K^\times$ being the set of invertible elements of $K$).     \\
Let $K$ be a differentially valued partial exponential field. Denote by val the valuation on $K$, $\Gamma$ its value group, and let $R$ be the valuation ring with maximal ideal $\cP$.  As we will see below, in many cases of interest one can extend the partial exponential on $K$ to a total exponential map (which is not uniquely determined though). Notable exceptions are Laurent power series fields, as well as generalized power series fields (whose definition will be recalled below).  It follows that by making the appropriate interpretation of each symbol of $\cL$, the field $K$ is then naturally an $\cL$-structure (with $\textrm{Exp}$ denoting the (extended) exponential map). \\
 Furthermore, the maximal ideal $\cP$ of $R$ can be defined as follows: 
$$ x \in \cP \quad {\rm iff} \quad x \in R \; \& \; x^{-1} \notin R. $$
Assume now that $R$ is a discrete valuation ring, and let $\pi \in R$ be a uniformizer, i.e. val$(\pi) =1$.  Let $\cL_\pi$  be the expansion $\cL \cup \{\pi\}$, with $\pi$ denoting a constant in $K$.  \\ 
Using $\pi$, the valuation val can be defined using the predicate $R$ in a standard way: ${\rm val}(x) \geq 0$ iff $R(x)$ ($x$ is in the valuation ring), and for all $x \in K$, val$(x) = n \in \mathbb{Z}$ iff val$(x/\pi^n) \geq 0 \; \&$ val$(x/\pi^n) \leq 0$.  \\ 
For any valued field $(K,{\rm val})$ (with a possibly non discrete, or even, a non-archimedean value group $G$ where we fix a (generally non-canonical) embedding $\mathbb{Z} \hookrightarrow G$ and identify $\mathbb{Z}$ with its image in $G$), one can still use the language $\cL_\pi$ (for some $\pi$ satisfying val$(\pi)=1 $), and in this case any set of the form $\{ x \in K \; | \; {\rm val}(x) > e\}, e \in \mathbb{Q}$ is $\cL_\pi$-definable.  Explicitly, the above set is defined through the formula (below $e=n/m$):
$$ \varphi_e(x): R(\frac{x^m}{\pi^n}) \; \& \; \neg R(\frac{\pi^n}{x^m}).     $$
The complex $p$-adic field $\mathbb{C}_p$ falls in particular in the above case: the value group of the standard valuation $\ordp$ is $\mathbb{Q}$, and any non-principal  ultraproduct $\prod_p \mathbb{C}_p/\cU$ (see below) has a non-archimedean value group. \\ 
From the above remarks, one can see that a formula of the form val$(x) = $ val$(y)$ is an abbreviation of $R(x/y) \; \& \; R(y/x)$. Note that the expressive power of the language $\cL$ falls short of defining every ball in $K$ (a set of the form val$(x-a) \geq g$ for $g \in G$ and some $a \in K$), since $g$ might be a non-standard element. 

An $\cL$-structure is a tuple $(K, R, \textrm{Exp})$, where $K$ is a valued field, $R$ its valuation ring and $\textrm{Exp}$ is an exponential map $\textrm{Exp}: K \to K^\times$.
\subsection{The field $\mathbb{K}_{\cU}$}
Let $\mathbb{P}$ be the set of prime numbers, and let $\cU$ be a non-principal ultrafilter on $\mathbb{P}$.  \\
 Here the predicate $R$ is interpreted as the set $\mathbb{C}^\circ_p$ of complex $p$-adic numbers with non-negative $p$-adic valuation. \\
Define the field $\mathbb{K}_\cU$ as the ultraproduct of the fields $\mathbb{C}_p$:
$$ \mathbb{K}_\cU := \prod_{p \in \mathbb{P}} \mathbb{C}_p / \cU. $$
The field $\mathbb{K}_{\cU}$ becomes an $\cL$-structure upon interpreting the function and predicate symbols in the standard way, for instance $R([(x_p)_{p \in \mathbb{P}}])$ if and only if the set of $p \in \mathbb{P}$ for which $x_p \in \mathbb{C}_p^\circ$ is in $\cU$. In this case we say that $x_p \in \mathbb{C}^\circ_p$ for $\cU$-almost all primes. \\
By application of \L o\'s Theorem on ultraproducts, $\mathbb{K}_{\cU}$ is shown to be an algebraically closed field equipped with the valuation induced by $\ordp$ (for $p$ running over $\mathbb{P}$).
Equip $\mathbb{K}_{\cU}$ with the valuation val defined as:
$$ {\rm val} ([x]) = [(\ordp(x_p))_{p \in \mathbb{P}}], $$
where we have used the notation $[x] := [(x_p)_{p \in \mathbb{P}}] \in \mathbb{K}_{\cU}$. The elements  $ (\ordp(x_p))_{p \in \mathbb{P}}$ belong to the Cartesian product of value groups $\prod_{p \in \mathbb{P}} \mathbb{Q}$, and $ [(\ordp(x_p))_{p \in \mathbb{P}}]$ belongs to the ultrapower of $\mathbb{Q}$, i.e. $\mathbb{Q}^\cU$. It is immediate to verify that val is indeed a valuation on $\mathbb{K}_{\cU}^{\times}$.  \\
For more details about ultraproducts of valued fields (and ultraproducts in general), see, e.g. ~\cite{S}.   \\
Let $k_\cU$ be the residue field, $k_\cU= R/P$, with $R$ and $P$ the valuation ring and its maximal ideal. It follows from \L o\'s Theorem that $k_\cU$ is an algebraically closed field of characteristic zero, hence $\mathbb{K}_\cU$ is an equicharacteristic valued field.   
\subsection{The exponential map}
Let $p$ be a prime number. 
 Fix an extension ${\rm EXP}_p$ of the $p$-adic exponential $\exp_p$ such that ${\rm EXP}_p$ is an exponential map defined for all elements of $\mathbb{C}_p$, i.e. ${\rm EXP}_p: \mathbb{C}_p \to \mathbb{C}_p^{\times}$ and
\begin{eqnarray*}
	\forall x \in \mathbb{C}_p, \abs{x}_p <p^{-1/p-1}, {\rm EXP}_p(x)  & = & \exp_p(x), \\
	\forall x, y \in \mathbb{C}_p \;\; {\rm EXP}_p(x+y) &  = & {\rm EXP}_p(x) {\rm EXP}_p(y).
\end{eqnarray*}
The existence of such an extension is guaranteed by Zorn Lemma (see ~\cite{R}  chap. 5, section 4.4). However, it is not unique. It can be seen that ${\rm EXP}_p$ is a continuous homomorphism from the additive group $(\mathbb{C}_p, +)$ to the multiplicative group $(\mathbb{C}^{\times}_p, \cdot)$.  \\ 
For each prime $p$, the field $\mathbb{C}_p$ equipped with the exponential map EXP$_p: \mathbb{C}_p \to \mathbb{C}_p^{\times}$ is a structure for $\cL$.  \\
Note that the use of the extension EXP$_p$ (rather than just the standard $p$-adic exponential $\exp_p$) seems to be useful from the model-theoretic point of view, in view of the intended application. More precisely, since we are considering an ultraproduct of the $\mathbb{C}_p$'s, the map $E([x]) := [(\exp_p(x_p))_p]$ (see below) is defined on an open disc around the origin of radius $1 -\epsilon$, with $\epsilon > 0$ an infinitesimal, whereas the domain of $\exp_p$ is the open disc of radius $r_p:=  p^{-1/(p-1)}$ as already observed. Using instead the maps EXP$_p$, allows us to have a uniform definition of the domain of the exponential map. 
\subsection{Ordered abelian groups}
 Let $(G, +, \leq)$ be an ordered abelian group under the law $+$, where $\leq $ denotes the order relation on $G$.  Let $G^{>0}$ be the semi-group of positive elements of $G$ (i.e. elements greater than $0$).    \\
Let $\Delta$ be the set of \textit{archimedean classes} of $G^{>0}$ (see, e.g. ~\cite{Hah}). The archimedean class of an element $g \in G$ will be denoted by $[g]$.  \\
If $\Delta$ is not a singleton, we say that $G$ is non-archimedean. The set $\Delta$
comes equipped with the inherited order $\preceq$ defined as:
$ \delta_1=[g_1] \preceq \delta_2=[g_2] \;\; {\rm iff} \; \; (\abs{g_2} \leq \abs{g_1}),$
for any $\delta_1, \delta_2 \in \Delta$. Obviously, we may define the induced relations $\prec$ and $\succ$ in a similar way. Let $[0]= \infty$. The order $\preceq$ can then be extended to $\Delta \cup \{\infty\}$ by setting $\delta \preceq \infty$ for all $\delta \in \Delta \cup \{\infty\}$.   \\
Denote by $v_1$ the map (called natural valuation) $v_1: G  \to \Delta \cup \{\infty \}$ defined as $v_1(g) = [g]$. 
\subsubsection{Hahn Embedding Theorem}
A central result in the theory of linearly ordered abelian groups is the following: \\
 Let $G$ be a linearly ordered abelian group.
Then there exists an embedding of ordered groups $i: G \hookrightarrow H(\Delta) \subset \mathbb{R}^{\Delta}$ where $\Delta$ is the set of archimedean classes of $G$, and $H(\Delta)$ (the {\it Hahn group} with respect to $\Delta $) is given by
 $$  H(\Delta) := \{a= (a_\gamma)_{\gamma \in \Delta}: a_\gamma \in \mathbb{R} \;  {\rm and}  \; {\rm Supp} (a) {\rm \;  is \;  well  \; ordered} \}.  $$
  Here $H(\Delta)$ is equipped with the lexicographic order, and Supp$(a)$ (for $a \in \mathbb{R}^{\Delta}$) is defined as
$$ {\rm Supp}(a) := \{ \gamma \in \Delta: a_\gamma \neq 0\}. $$
Any element $g$ of $G$ can be written as $g = \sum_{\phi \in \Delta} g_\phi \mathbf{1}_\phi$ where $g_\phi \in \mathbb{R}$ and ${\mathbf{1}}_\phi, \phi \in \Delta$ the element of $\Gamma$ that corresponds through the embedding $i$ to $(a_\psi)_{\psi \in \Delta} \in \mathbb{R}^\Delta$, with $a_\phi=1$ $a_\psi=0$, for $\psi \neq \phi$. We have $ v_1(g) = \min(\rm Supp(g)) \in \Delta \cup \{\infty\}$. \\
Let $\Gamma$ be the value group of $\mathbb{K}_\cU$, $\Gamma:= (\prod_{p \in \mathbb{P}} \mathbb{Q} /\cU, +)$.   \\
Observe that we have a canonical embedding $\mathbb{Q} \hookrightarrow \Gamma$, $r \mapsto [(r_p)_{p \in \mathbb{P}}]$ (with $r_p =r$ for all $p \in \mathbb{P}$).
 An element of $\Gamma$ is called standard if it is in the image of $\mathbb{Q}$ by this embedding. \\
Let $\gamma:= [(g_p)_{p \in \mathbb{P}}]$ be an element of $\Gamma$ such that for any $\varepsilon >0$, there exists $p_0 \in \mathbb{P}$ for which $\forall p \in \mathbb{P}, \, p >p_0 \Rightarrow \! \abs{g_p} < \varepsilon$. Then clearly, $\gamma$ is an infinitesimal element, since it is smaller (in absolute value) than any element of $\mathbb{Q}^{>0}$. Similarly, an element $[(g_p)_{p \in \mathbb{P}}]$ of $\Gamma$ is infinite iff it satisfies
$$  \forall A \in \mathbb{Q}^{>0},  \exists p_0 \in \mathbb{P} (p > p_0 \rightarrow \abs{g_p} > A). $$
Note that the above definitions are not first-order, since we have no way of quantifying over standard positive rationals in the language. 
 We have: 
 \begin{prop}
   The group $\Gamma$ is an ordered abelian group. Furthermore, the set $\Delta$ of archimedean classes of $\Gamma$ is an infinite, unbounded, densely linearly ordered set having uncountable cofinality.  
 \end{prop} 
\begin{proof} 
	The first assertion follows using standard properties of ultraproducts, e.g. ~\cite{AP}.
	To see that $\Delta$ is unbounded we equip $\Gamma$ with the multiplicative operation (compatible with the order) induced by standard multiplication on $\mathbb{Q}$, endowing $\Gamma$ with an ordered field structure. Hence $\Delta$ acquires a group structure through $v_1(\alpha) + v_1(\beta)= v_1(\alpha \beta)$, for all $\alpha, \beta \in \Gamma^{>0}$, where $v_1: \Gamma \to \Delta \cup \{\infty \}, g \mapsto [g]$ as defined above. \\
    It follows that $\Delta$ is the value group for the natural valuation $v_1$. Now the required conclusion follows since $\Gamma$ is nonarchimedean.   \\ 
	Let us now show that $\Delta$ has uncountable cofinality.    \\ 
	Assume there exists a countable sequence $(\delta_n)_{n \in \mathbb{N}}, \delta_n \in \Gamma$ such that 
	$$ \forall \alpha \in \Gamma, \exists n_0 \in \mathbb{N}, \forall n > n_0, [\alpha] \preceq [\delta_n].  \qquad \qquad   (\dagger)$$
	Writing $\delta_n=[(\delta_{nk})_{k \in \mathbb{P}}]$, $\delta_{nk} \in \mathbb{Q}_+$, one can check that the double sequence $(\delta_{nk})_{n \geq 0, k \in \mathbb{P}}$ is strictly increasing (beyond some $n_0, k_0$). Now let $\alpha:= [(\alpha_k)_k]$ be defined by $\alpha_k = \delta_{kk}$.  \\
	Then it can be checked that $(\dagger)$ does not apply for $\alpha$, contradiction.  Since $\Gamma$ has cardinality $2^{\aleph_0}$, the cofinality of $\Delta$ is at most $2^{\aleph_0}$. \\
	Finally, to see that $\Delta$ is densely ordered, assume the contrary. It suffices then to observe that the induced order on the set $\Gamma^{\geq 0}$  
	is of type $>\omega$, on which there exists no possible cancellative semi-group structure compatible with the ordering. This contradiction proves the result.  	  
\end{proof}
\subsubsection{Kaplansky embedding theorem} 
Let $\Gamma$ be an ordered abelian group and $k$ a commutative field. Let $k((t^\Gamma))$ be the field of generalized power series with a well-ordered 
set of exponents in $\Gamma$ and coefficients in $k$. Denote by $v$ the $t$-adic valuation of $k((t^\Gamma))$.  
We are now able to apply the following:
\begin{theorem} \label{K} (Kaplansky ~\cite{K}) 
	Let $(K, \textrm{val})$ be a valued field of zero equi-characteristic, with value group $\Gamma$ and algebraically closed residue field $k$.
	Then $(K, \textrm{val})$ is \textbf{analytically isomorphic} to a subfield of $(k((t^\Gamma)), v)$, i.e. there exists a value preserving embedding of fields $K \hookrightarrow k((t^\Gamma))$.
\end{theorem}
 The original statement in ~\cite{K} is more general, allowing non-algebraically closed residue fields at the expense of introducing \textit{factor sets} into the definition of the multiplicative operation of monomials in the power series field. By Theorem 7 of ~\cite{K}, this turns out not to be necessary in the special case of an algebraically closed residue field. 
\section{A differential exponential valued field}
 Now we consider again  the field $\mathbb{K}_{\cU}$. Observe that the valuation ring of $\mathbb{K}_{\cU}$ is given by
$$ R = \{ x:=[(x_p)_{p \in \mathbb{P}}] : {\rm val}(x) \geq 0 \},$$
where the order relation (on the value group of $\mathbb{K}_{\cU}$) has already been explained in the previous section. \\
We can easily show that the residue class field $k_\cU$ ($k_\cU = R/P$ where $P$ is the maximal ideal of $R$) is given by
$$ k_\cU = \prod_{p \in \mathbb{P}} \mathbb{F}_p^{\rm alg} /\cU,$$
where $\mathbb{F}_p^{\rm alg}$ is the algebraic closure of the finite field $\mathbb{F}_p$.  By Lefshetz principle (see, e.g. Theorem 2.4.3 ~\cite{S}) we have $k_\cU \simeq \mathbb{C}$, since both are algebraically closed fields of cardinality $2^{\aleph_0}$ having characteristic zero.   \\
Applying Kaplansky's result mentioned above, there exists an embedding of valued fields $\mathbb{K}_{\cU} \hookrightarrow \mathbb{L}_{\cU} := k_\cU((t^\Gamma))$. For each non-principal ultrafilter $\cU$ over $\mathbb{P}$, we fix an embedding $\iota_\cU$
$$ \iota_\cU: \mathbb{K}_{\cU} \hookrightarrow \mathbb{L}_{\cU}, $$
and we will denote by $v$ the canonical valuation on $\mathbb{L}_\cU$.  \\
The $p$-adic exponential map on each $\mathbb{C}_p$ can be used to introduce a total exponential map on $\mathbb{K}_{\cU}$.  \\
More precisely, one may show (using \L o\'s Theorem) that the map $\textrm{Exp}: [(x_p)_p]  \mapsto [({\rm EXP}_p(x_p))_p]$ is indeed an exponential map, $\mathbb{K}_{\cU} \to \mathbb{K}_\cU^{\times}$ (satisfying $\textrm{Exp}(x+y) =\textrm{Exp}(x) \cdot \textrm{Exp}(y)$).  
\subsubsection{An exponential differential field}  \label{expdif}
In order to be able to apply Ax's Theorem, we need to define an embedding of $\mathbb{K}_\cU$ into a (partial) exponential  differentiable field, along the lines of ~\cite{KM} and ~\cite{KMS} (see also ~\cite{M} for a general survey). As will be seen, this embedding need not be an embedding of \textit{differential fields}, neither this is assumed.  \\
First we define a right-shift map $\sigma: \Delta \to \Delta, \phi \mapsto \sigma(\phi)$ such that $\sigma(\phi) \succ \phi$ and $\sigma$ is order-preserving.  \\
Let $\delta$ be the archimedean class of some infinitesimal element of $\Gamma$. We set:
$$ \sigma: \Delta \to \Delta, \phi \mapsto \delta \cdot \phi. $$
 Here by $\delta \cdot \phi$ we mean the archimedean class of any product of two elements in $\delta$ and $\phi$ respectively. It can be seen that this is independent of the choices, and that, indeed $\sigma(\phi) \succ \phi, \forall \phi \in \Delta$.  \\
Then, as in "Case 1" of Example (6) in ~\cite{KMS}, one may define a derivation $D: \mathbb{L}_{\cU} \to \mathbb{L}_{\cU}$ with field of constants $k_\cU$,
and which satisfies furthermore: $Dx= \frac{D(\exp(x))}{\exp x}$ since D is a series derivation (see ~\cite{H}, Corollary (3.9)).  \\
Let us denote by $\mathbb{L}_{\cU}^{\circ}$ the ring of bounded elements of $\mathbb{L}_{\cU}$, and by  $\mathbb{L}_{\cU}^{\circ \circ}$ its maximal ideal, i.e. the ideal of infinitesimal elements. 
 Note that $\mathbb{L}_{\cU}^{\circ \circ} = k_{\cU}((t^{\Gamma^{>0}}))$ (the set of generalized power series with strictly positive support).    \\
Let $\cD_\cU$ be the set defined as:
$$  \cD_\cU := \bigg\{ x=[(x_p)_{p \in \mathbb{P}}] \in \mathbb{K}_{\cU} \; : \; \ordp(x_p) > \frac{1}{p-1} \; \textrm{for} \; \cU\!\!-\!\textrm{almost all} \;  p \in \mathbb{P} \bigg\}.$$
Consider the map $E: \cD_\cU \rightarrow \mathbb{K}_{\cU}^{\times} \subset \mathbb{L}_{\cU}^{\times}$ defined by
$$ [(x_p)]_{p \in \mathbb{P}} \mapsto E( [(x_p)]_{p \in \mathbb{P}}) := [(\exp_p(x_p))_{p \in \mathbb{P}}]. $$
Using \L \'os Theorem we see that $E$ is a partial exponential map on $\mathbb{L}_{\cU}$ (i.e. $E(x+y)= E(x) E(y)$).  \\
In what follows we note that, using Neumann Lemma (see ~\cite{N}), the series $\sum_{n \in \mathbb{N}} \frac{x^n}{n!}$ is summable for all $x \in k((t^{G^{>0}}))$, for any field $k$ and ordered abelian group $G$, and $\exp(x) \in k((t^G))$. In particular the map $\exp: \mathbb{L}_\cU^{\circ \circ} \to \mathbb{L}_\cU^\times, x \mapsto \exp(x):= \sum_{n \in \mathbb{N}} \frac{x^n}{n!}$ is well defined.  
\begin{theorem}
	The map $E$ coincides with the map $x \mapsto \exp(x)= \sum_{n \in \mathbb{N}} \frac{x^n}{n!}$ on $\cD_\cU$, i.e. $E(x)$ is given by the Taylor formula for the standard exponential map. 
	In other words, the embedding $\iota_\cU: \mathbb{K}_{\cU} \hookrightarrow \mathbb{L}_{\cU}$ commutes with the exponential, i.e. $\iota_\cU(E(x)) = \exp(\iota_\cU(x))$ for all $x \in \cD_\cU$. 
\end{theorem}
\begin{proof}
  Let $\alpha \in \cD_\cU$. Then $\alpha=[(\alpha_p)_p]$, and $\ordp(\alpha_p) > \frac{1}{p-1}$ for $\cU$-almost all $p$. Consider the fields $F_0= k_{\cU}(\alpha)$ and $F= k_\cU(\alpha, E(\alpha))$ and let $\Gamma_0$ be the divisible hull of the group val$(F^{\times})$. By (~\cite{AP} Theorem 3.4.3) 
   $\Gamma_0$ is an ordered abelian group having finite dimension as a linear space over $\mathbb{Q}$. This will be shown directly below. \\ 
  From the embedding $F  \hookrightarrow \mathbb{L}_{\cU}$ we get an embedding of valued fields $F \hookrightarrow k_\cU((t^{\Gamma_0}))$ such that the following diagram commutes
  \begin{equation*}
  \xymatrix{
  	F  \ar[d] \ar[r] & \mathbb{L}_{\cU} \\
  	k_\cU((t^{\Gamma_0}))  \ar[ur]  &
  }
  \end{equation*}
  obtained by identifying $F$ with its image in $\mathbb{L}_{\cU}$.   \\ 
  Let $s_N(\alpha)$ be the partial sum $s_N(\alpha):= \sum_{n=0}^N \frac{\alpha^n}{n!}$. 
  For every $p$, the sequence $(s_N(\alpha_p))_N$ is a Cauchy sequence in $\mathbb{C}_p$, hence in particular it is pseudo-Cauchy, and $\exp_p(\alpha_p)$ is a (pseudo-)limit. 
   Thus we obtain using \L \'os Theorem that for all positive integers $N_1 < N_2 <N_3$ 
   $$ \mathbb{K}_{\cU} \models \textrm{val}(s_{N_3}(\alpha) -s_{N_2}(\alpha)) > \textrm{val}(s_{N_2}(\alpha) -s_{N_1}(\alpha)), $$
   (where we used the abbreviation val defined in section \ref{lang})
   and $(s_N(\alpha))_N$ is pseudo-Cauchy in $\mathbb{K}_\cU$. In particular,  $(s_N(\alpha))_N$ is pseudo-Cauchy in $F_0$. 
   Also, we have 
    $$ \mathbb{K}_{\cU} \models \textrm{val}(E(\alpha) -s_{N}(\alpha)) = \textrm{val}(s_{M}(\alpha) -s_N(\alpha)), $$
    (by \L \'os Theorem) for all sufficiently large $N$, and all  $M>N$. \\
   It follows by ~\cite{K} (Theorems 2 and 3) that the extension $F_0 
   \hookrightarrow F_0(E(\alpha)) = F$  
   is an immediate valued field extension, and $\Gamma_0 = \mathbb{Q} \cdot \gamma$ where $\gamma= \textrm{val}(\alpha)$, hence $\Gamma_0$ has finite dimension as a linear space over $\mathbb{Q}$, as claimed. It follows in particular that the field $k_\cU((t^{\Gamma_0}))$ is Hausdorff and complete with respect to the topology induced by the valuation $v_{|k_\cU((t^{\Gamma_0}))}$. \\ 
   For any $x \in \cD_\cU$, one has 
   $$ \mathbb{C}_p \models \ordp(\exp_p(x_p) -s_N(x_p)) = \ordp\bigg(\frac{x_p^{N+1}}{(N+1)!}\bigg) $$
   for all $N$, for $\cU$-almost all $p$. Hence,  
   in this case we have:
   $$ \mathbb{K}_\cU \models \textrm{val}(E(\alpha) - s_N(\alpha)) = \textrm{val}\bigg(\frac{\alpha^{N+1}}{(N+1)!}\bigg) $$
   for all $N$.
    
   From the above observation we see that the sequence val$(\alpha^N)$ is cofinal in $\Gamma_0$, hence $\alpha^N \to 0$ as $N \to \infty$ (in $F$) and the sequence $E(\alpha)-s_N(\alpha)$ converges to zero in $k_{\cU}((t^{\Gamma_0}))$. Also, in $k_\cU((t^{\Gamma_0}))$ we have that $s_N(\alpha) \rightarrow \exp(\alpha)$. Consequently, $E(\alpha)=\exp(\alpha)$ as required.   
\end{proof}
 Using the above, we reach the following corollary: 
 \begin{cor}
 	For all $x \in \cD_\cU$, one has: $D(E(x))= E(x)Dx$. 
 \end{cor}
\begin{proof}
	This follows from $E(x) = \sum_{n \in \mathbb{N}} \frac{x^n}{n!}= \exp(x)$ and that $Dx= \frac{D(\exp(x))}{\exp(x)}$ as observed in subsection \ref{expdif}.
\end{proof}
\subsection{Proof of Theorem \ref{main}}  \label{Proofs}
 In this section we apply the above considerations in order to obtain Theorem \ref{main}.   \\
Let $V$ be a variety of dimension $n$ in the affine $2n$-space $\mathbb{A}^{2n}_{\mathbb{Q}}$. For each prime $p$, denote by $W(\mathbb{C}_p)$ the set of tuples $\bar{a}_p \in E_p^n$ for which $(\bar{a}_p, \exp_p(\bar{a}_p)) \in V(\mathbb{C}_p)$.  \\
Let $S' \subset \mathbb{P}$ denote the set of primes $p$ for which $V$ has a $\mathbb{C}_p$-point of the form $(\bar{a}_p, \exp_p(\bar{a}_p))$ (hence, in particular, $\abs{\bar{a}_p}_p < p^{-1/p-1}$).
The following will be assumed throughout:  \\
$(\star)$ There are infinitely many primes $p$ such that  $V$ has a $\mathbb{C}_p$-point of the form $(\bar{a}_p, \exp_p(\bar{a}_p))$. In other words, $S'$ is an infinite set. \\
The proof of Theorem \ref{main} proceeds by assuming the contrapositive. More precisely, let $(\dagger)$ denote the following statement:  \\
$(\dagger)$  There exists no {\it uniform} rational linear dependence that holds for $\bar{a}_p$, for infinitely many $p \in S'$.  \\ 
      Let $\cU$ be a non-principal ultrafilter on $\mathbb{P}$, such that $S' \in \cU$. 
	   Define $\bar{x} =(x_1, \dots, x_n) = [(\bar{x}_p)_p] \in \mathbb{K}_\cU^n \subset \mathbb{L}_\cU^n$ as follows: \\
	   If $p \in S'$, then $\bar{x}_p = \bar{a}_p$. Otherwise, we let $\bar{x}_p$ be an arbitrary $n$-tuple of complex $p$-adic numbers which lie in the domain of the corresponding $\exp_p$ (this last assumption, though harmless, is not strictly necessary). 
	   Then $\bar{x}$ is  an $n$-tuple of elements of $\cD_\cU$. \\
	   Applying \L \'os Theorem to $\mathbb{K}_\cU$ it follows that $(\bar{x}, \textrm{Exp}(\bar{x})) \in V(\mathbb{K}_\cU)$, consequently $(\bar{x}, E(\bar{x})) \in V(\mathbb{K}_\cU) \subset V(\mathbb{L}_\cU)$. Hence, applying Proposition \ref{ASP} to the (partial) exponential valued field $\mathbb{L}_\cU$, it follows that
$$ m_1 x_1 + \dots + m_n x_n \in k_{\cU}, $$
for some $m_1, \dots, m_n \in \mathbb{Q}$ (not all zero). Clearing denominators, we can assume that $m_1, \dots, m_n$ are integers. \\
 Furthermore, since the elements $x_1, \dots, x_n$ lie in the maximal ideal $\mathbb{L}_\cU^{\circ \circ }$, any $\mathbb{Z}$-linear combination of $x_1, \dots, x_n$ will necessarily be in $\mathbb{L}_\cU^{\circ \circ}$.
 Writing $x_i = [(x_{p,i})_p]$ for $i=1, \dots, n$, it follows $$m_1 x_{p,1} + \dots + m_n x_{p,n} =0, \qquad \qquad (*)$$ for $\cU$-almost all $p \in \mathbb{P}$, i.e., the set of primes $p$ for which $(*)$ holds belongs to the ultrafilter $\cU$. In particular, observing that the intersection of any two sets in $\cU$ is in $\cU$ (so it is an infinite set, $\cU$ being a non-principal ultrafilter), one has, for infinitely many primes $p$ 
 $$ m_1 a_{p,1} + \dots + m_n a_{p,n} =0,   \qquad \qquad \qquad \qquad (\ddagger)$$
 with fixed $m_1, \dots, m_n$ (not all zero), contradicting the hypothesis.  \\
  In particular it follows from the above argument that the set of rational $n$-tuples $(m_1, \dots, m_n)$ (up to a non-zero multiplicative constant) for which $(\ddagger)$ holds uniformly for infinitely many primes $p \in S'$ is finite. Let us denote this set by $A$.  \\
  In order to show the remaining statement of Theorem \ref{main}, let us define $\bar{\alpha}:= (\bar{\alpha}_p)_p$ (i.e. $\bar{\alpha}$ defines a family of $n$-tuples of complex $p$-adic numbers that belong to $W(\mathbb{C}_p)$ for infinitely many $p$'s) and denote by $S_{\bar{\alpha}}$ the set of primes for which $(\ddagger)$ does not hold for the tuple $\bar{\alpha}_p$ for any $(m_1, \dots, m_n) \in A$. 
  By the above reasoning, $S_{\bar{\alpha}}$ is a finite set. Assume, for the sake of contradiction, that there is no finite set $S$ for which $S_{\bar{\alpha}} \subset S$ for all $\bar{\alpha}$. Choose a countable subset of these  $\bar{\alpha}$ enumerated as $\bar{\alpha}_1, \bar{\alpha}_2, \dots$, such that $\bigcup_{i=1}^\infty S_{\bar{\alpha}_i} \subset \mathbb{P}$ is an infinite set. Without loss of generality the sets $S_{\bar{\alpha}_i}$ can be assumed disjoint. One can then construct an infinite family of tuples $(\bar{\beta}_p)_p$ which satisfy $(\dagger)$: set $\bar{\beta}_p = \bar{\alpha}_{i,p}$ for $p \in S_{\bar{\alpha}_i}$, and assign arbitrary values to $\bar{\beta}_p$ for $p \notin S_{\bar{\alpha}_i}, \forall i$. Repeating the above reasoning, we reach a contradiction. This proves the claim.   \\
   Combining the above results, we see that Theorem ~\ref{main} is now fully proved.

\end{document}